\documentclass{amsart}
\usepackage{amssymb}
\usepackage{eucal}
\usepackage{amsfonts}
\usepackage{epsfig}
\usepackage[all]{xy}
\newtheorem{thm}{Theorem}[section]
\newtheorem{cor}[thm]{Corollary}
\newtheorem{lem}[thm]{Lemma}
\newtheorem{prop}[thm]{Proposition}
\theoremstyle{definition}

\theoremstyle{remark}

\numberwithin{equation}{section} 
\newtheorem{ex}[thm]{Example}

\newcommand{\bbM}{\mathbb{M}}

\newcommand{\mcG}{\mathcal{G}}
\newcommand{\mcC}{\mathcal{C}}
\newcommand{\mcD}{\mathcal{D}}

\newcommand{\bfg}{\mathbf{g}}
\newcommand{\bfa}{\mathbf{a}}

\newcommand{\bfx}{\mathbf{x}}
\newcommand{\uHom}{\underline{\Hom}}

\newcommand{\Ob}{\operatorname{Ob}}
\newcommand{\id}{\operatorname{id}}
\newcommand{\Hom}{\operatorname{Hom}}
\newcommand{\oH}{\operatorname{H}}
\newcommand{\oC}{\operatorname{C}}

\newcommand{\act}[2]{{}^{#1}\!{#2}}

\newcommand{\ad}{\operatorname{Ad}}

\def\vvdots{\setbox0=\hbox{$\vdots$}\ht0=8pt\box0}

\let\:=\colon

\def\Un{ \mbox{Un} }

\def\bX{ \bar{X}}
\def\Ext{{ \mbox{Ext} }}

\def\Spec{{ \mbox{Spec} }}

\def\Hom{{ \mbox{Hom} }}

\def\ra{{ \rightarrow }}

\def\a{{ \alpha }}
\def\b{{ \beta} }
\def\g{{ \gamma }}
\def\d{{ \delta }}

\def\F{ {\mathbb F} }
\def\hZ{ \hat{\Z}}

\def\Un{ \mbox{Un} }

\def\z{\zeta}

\def\C{{ \mathbb{C} }}

\def\bs{ \backslash}

\def\G{{ \Gamma }}
\def\Gal{{ \mbox{Gal} }}

\def\bF{{ \bar{F} }}

\def\Z{{ \mathbb{Z}}}

\def\bq{\begin{quote}}
\def\eq{\end{quote}}
\def\Aut{ \mbox{Aut}}

\def\Q{\mathbb{Q}}

\def\invlim{\varprojlim}

\def\P{ \mathbb{ P}}

\def\be{\begin{equation}}
\def\ee{\end{equation}}

\def\bF{ \bar{F}}

\def\tf{ \tilde{f}}

\def\bF{ \bar{F}}

\def\b{ \bar}

\def\L{\Lambda}

\def\a{\alpha}

\def\b{\beta}

\def\bs{\begin{slide}}
\def\es{\end{slide} }

\def\Gm{\mathbb{G}_m}

\numberwithin{equation}{section}

\def\ad{\mbox{ad}}

\def\Ker{\mbox{Ker}}
\def\beq{\begin{equation}}
\def\eeq{\end{equation}}
\def\O{\mathcal{O}}

\usepackage{graphicx}

\def\Ad{\mbox{Ad}}

\def\ext2et{\Ext^2_{\Un(\bX)}}

\def\ms{\medskip}

\def\cK{{\mathcal K}}

\def\inv{\mbox{inv}}

\def\bF{\bar{ F}}
\def\<{\langle}
\def\>{\rangle}
\def\Zn{\Z/n\Z}
\def\xx{{\times }}
\def\msloc{\mathcal{M}^S_{loc}}
\def\nZ{\frac{1}{n}\Z/\Z}

\def\Ainfty{A^{\infty}}
\title{Arithmetic Chern-Simons Theory I }
\author{Minhyong Kim}
\thanks{M.K. Supported by  grant 	EP/M024830/1 from the EPSRC}
\subjclass{14G10, 11G40, 81T45
}

\address{M.K: Mathematical Institute, University of Oxford and Department of Mathematics, Ewha Womans University}
\address{B.N.: Department of Mathematics, Queen Mary University of London}
\begin{document}
\maketitle
\begin{flushright}
with Appendix B by Behrang Noohi
\end{flushright}
\begin{abstract}In this paper, we apply   ideas of Dijkgraaf and Witten \cite{witten, DW} on 2+1 dimensional topological quantum field theory to arithmetic curves, that is, the spectra of rings of integers in algebraic number fields. In the first three sections, we define  classical Chern-Simons functionals on spaces of Galois representations. In the highly speculative section 6, we consider the far-fetched possibility of using Chern-Simons theory to construct $L$-functions.
\end{abstract}

\begin{flushleft}

\section{The arithmetic Chern-Simons action: basic case}

We wish to move rather quickly to a concrete definition in this first section. The reader is directed to section 5 for a motivational discussion of $L$-functions.
\ms

Let $X=\Spec(\mathcal{O}_F)$, the spectrum of the ring of integers in a number field $F$. We  assume that $F$ is totally imaginary, for simplicity of exposition. Denote by $\Gm$ the \'etale sheaf that associates to a scheme the units in the global sections of its coordinate ring. We have the following canonical isomorphism (\cite{mazur}, p. 538):
$$\inv: H^3(X, \Gm)\simeq \Q/\Z. \ \ \ \ \ (*)$$
This map is deduced from the `invariant' map of local class field theory. We will use the same name for a  range of isomorphisms having the same essential nature, for example,
$$\inv:H^3(X, \Z_p(1))\simeq \Z_p,\ \ \ \ \ (**)$$
where $\Z_p(1)=\invlim_i \mu_{p^i}$, and  $\mu_n\subset \Gm$ is the sheaf of $n$-th roots of 1. This follows from  the exact sequence
$$0\rightarrow \mu_n\rightarrow \Gm
\stackrel{(\cdot)^n}{\rightarrow}\Gm\rightarrow \Gm/(\Gm)^n\rightarrow 0.$$
 That is, according to loc. cit., $$H^2(X,\Gm)=0,$$ while by op. cit., p. 551, we have 
$$H^i(X,\Gm/(\Gm)^n)=0$$ for $i\geq 1$. If we break up the above into two short exact sequences,
$$0\rightarrow \mu_n\rightarrow \Gm\stackrel{(\cdot)^n}{\rightarrow}\cK_n\rightarrow 0,$$
and
$$0\rightarrow \cK_n \rightarrow \Gm\rightarrow \Gm/(\Gm)^n\rightarrow 0,$$
we deduce $$H^2(X, \cK_n)=0,$$from  which it follows that
$$H^3(X, \mu_n)\simeq \nZ,$$
the $n$-torsion inside $\Q/\Z$. Taking the inverse limit over $n=p^i$ gives the second isomorphism above.  The pro-sheaf $\Z_p(1)$ is a very familiar coefficient system for \'etale cohomology and (**) is  reminiscent of the fundamental class of a compact oriented three manifold for singular cohomology. Such an analogy was noted by Mazur around 50 years ago \cite{mazur2} and has been developed rather systematically by a number of mathematicians, notably, Masanori Morishita \cite{morishita}. Within this circle of ideas is included the analogy between knots and primes, whereby the map
$$\Spec(\mathcal{O}_F/\mathfrak{P}_v)\rightarrowtail X$$
from the residue field of a prime $\mathfrak{P}_v$ should be similar to the inclusion of a knot.  Let $F_v$ be the completion of $F$ at the place $v$ and $\mathcal{O}_{F_v}$ its valuation ring. If one takes this analogy seriously (as did Morishita), the map $$\Spec(\mathcal{O}_{F_v})\rightarrow X,$$ should be similar to the inclusion of a handle-body around the knot, whereas $$\Spec(F_v)\rightarrow X$$ resembles the inclusion of its boundary torus\footnote{It is not clear to me that the topology of the boundary should really be a torus. This is reasonable if one thinks of the ambient space as a three-manifold. On the other hand, perhaps it's possible to have a notion of a knot in a {\em homology three-manifold} that has an exotic tubular neighbourhood? }. Given a finite set $S$ of primes, we can look at the scheme
$$X_S:=\Spec(\mathcal{O}_F[1/S])=X\setminus \{\mathfrak{P}_v\}_{v\in S}.$$
Since a link complement is homotopic to the complement of a tubular neighbourhood, the analogy is then forced on us between $X_S$ and a  three manifold with boundary given by a union of tori, one for each `knot' in $S$. These of course are basic morphisms in 2+1 dimensional topological quantum field theory \cite{atiyah}. From this perspective, perhaps the coefficient system $\Gm$ of the first isomorphism should have reminded us of the $S^1$-coefficient important in Chern-Simons theory \cite{witten, DW}. A more direct analogue of $\Gm$ is the sheaf $\mathcal{O}_M^\xx$ of invertible analytic functions on a complex variety $M$. However, for compact Kaehler manifolds, the comparison isomorphism $$H^1(M, S^1)\simeq H^1(M, \mathcal{O}_M^\xx)_0,$$where the subscript refers to the line bundles with trivial Chern class, is a consequence of Hodge theory. This indicates that in the \'etale setting with no natural constant sheaf of $S^1$'s, the familiar $\Gm$ has a topological nature, and can be regarded as a substitute\footnote{Recall, however, that it is of significance in Chern-Simons theory that one side of this isomorphism is purely topological while the other has an analytic structure.}. One problem, however, is that  the $\Gm$-coefficient computed directly gives divisible torsion cohomology, whence the need for considering coefficients like $\Z_p(1)$ in order to get functions of geometric objects having an analytic nature as arise, for example, in the theory of torsors for motivic fundamental groups \cite{CK, kim1, kim2, kim3, kim4}.
\ms

Let $$\pi=\pi_1(X, b),$$  the profinite \'etale fundamental group of $X$,
where we take $$b: \Spec(\bF)\rightarrow X$$ to be the geometric point coming from an algebraic closure of $F$. 
Assume now that the group $\mu_n(\bF)$ of $n$-th roots of 1 is in $F$. Fix an isomorphism $\zeta_n:\nZ\simeq \mu_n$. Then
$$\inv: H^3(X, \Z/n\Z)\simeq H^3(X,\mu_n)\simeq \frac{1}{n}\Z/\Z.$$
Now let $A$ be a finite group and fix a class $c\in H^3(A, \Z/n\Z)$.
Let $$\mathcal{M}(A):= \Hom_{cont}(\pi, A)/A$$
be the set of isomorphism classes of principal $A$-bundles over $X$. Here, the subscript refers to continous homomorphisms, on which $A$ is acting by conjugation. For $[\rho]\in \mathcal{M}(A)$, we get a class $$\rho^*(c)\in H^3(\pi, \Z/n\Z)$$ that depends only on the isomorphism class $[\rho]$. Denoting by $\inv$ also the composed map $$H^3(\pi, \Z/n\Z) \rightarrow H^3(X, \Z/n\Z)\simeq \frac{1}{n}\Z/\Z.$$
We get thereby a function
$$CS_c: \mathcal{M}(A)\ra \nZ;$$
$$ [\rho]\mapsto \inv(\rho^*(c)).$$
 This is the basic and easy case of the classical Chern-Simons functional in the arithmetic setting.
\ms

Examples might be constructed along the following lines. Let $A=\Z/n\Z$, $\a\in H^1(A,\Z/n\Z)$ the class of the identity, and $\b\in H^2(A, \Z/n\Z)$ the class of the extension
$$0\rightarrow \Zn \stackrel{n}{\rightarrow} \Z/n^2 \Z\rightarrow \Zn\rightarrow 0.$$
Then $\b=\d\a$, where $\d:H^1(A, \Zn)\rightarrow H^2(A,\Zn)$ is the boundary map arising from the extension.
From the cohomology theory of finite cyclic groups (\cite{NSW}, I.7), we know that
$$(\cdot)\cup \b: H^1(A,\Zn)\rightarrow H^3(A,\Zn)$$
is an isomorphism. Put $$c:=\a\cup \b=\a\cup \d \a\in H^3(A, \Zn).$$
Then
$$CS_c([\rho])=\inv[ \rho^*(\a) \cup \d \rho^*(\a)],$$
in close analogy to the formulas of abelian Chern-Simons theory.
\section{The arithmetic Chern-Simons action: boundaries}

Let $n$ be a natural number and $S$ a finite set of primes in $\mathcal{O}_F$. We assume in this section that all primes of $F$ dividing $n$ are in $S$.
Let $$\pi_S:=\pi_1(X_S, b)$$ and $$\pi_v=\Gal(\bF_v/F_v),$$ equipped with maps
$$i_v: \pi_v\rightarrow \pi_S$$
given by choices of embeddings $\bF\rightarrowtail \bF_v$.  The collection $$\{i_v\}_{v\in S}$$ will be denoted by $i_S$.
Let
$$Y_S(A):=\Hom_{cont}(\pi_S, A)$$ and denote by
$\mathcal{M}_S(A)$ the action groupoid whose objects are the elements of $Y_S(A)$ with morphisms given by the conjugation action of $A$. We also have the local version
$$Y_S^{loc}(A)=\prod_{v\in S} \Hom_{cont}(\pi_v, A)$$
as well as the action groupoid $\mathcal{M}^S_{loc}(A)$  with objects 
$Y_S^{loc}(A)$ and morphisms given by the action of $A^S:=\prod_{v\in S} A$
conjugating the separate components in the obvious sense.
 Thus, we  have the restriction functor $$r_S: \mathcal{M}_S(A)\rightarrow \mathcal{M}_S^{loc}(A),$$
 where a homomorphism $\rho: \pi_S\rightarrow A$ is restricted to the collection
 $$i_S^*\rho:=(\rho\circ i_v)_{v\in S}$$ and $A$ is embedded diagonally in $A^S$.
\ms

We will now employ a cocycle $c\in Z^3(A, \Z/n\Z)$ to associate a $\nZ$-torsor  to each point of $Y^S_{loc}(A)$ in an $A^S$-equivariant manner. This will be  a finite arithmetic version of the Chern-Simons line bundle \cite{FQ}   over $\mathcal{M}^S_{loc}$.
We use the notation
$$C^i_S:=\prod_{v\in S} C^i(\pi_v, \Z/n\Z)$$
for the continuous cochains,
$$Z^i_S:=\prod_{v\in S} Z^i(\pi_v, \Z/n\Z)\subset C^i_S$$
for the cocycles, and
$$B^i_S:=\prod_{v\in S} B^i(\pi_v, \Z/n\Z)\subset Z^i_S\subset C^i_S$$
for the coboundaries.
In particular, we have the coboundary map (see Appendix A for the sign  convention)
$$d:C^2_S\rightarrow Z^3_S.$$

Let $\rho_S:=(\rho_v)_{v\in S}\in Y_S^{loc}(A)$ and put
$$c\circ \rho_S:=(c\circ \rho_v)_{v\in S},$$ $$c\circ Ad_a:=(c\circ Ad_{a_v})_{v\in S}$$
 for $a=(a_v)_{v\in S}\in A^S$, where $Ad_{a_v}$ refers to the conjugation action.  To define the arithmetic Chern-Simons line associated to $\rho_S$, we need the intermediate object
$$H(\rho_S):=d^{-1}(c\circ \rho_S)/B^2_S\subset C^2_S/B^2_S.$$
This is a torsor for
$$H^2_S:=\prod_{v\in S} H^2(G_v, \Z/n\Z)\simeq \prod_{v\in S} \nZ.$$
(\cite{NSW},  Theorem (7.1.8).)
We then use the sum map
$$\Sigma: \prod_{v\in S} \nZ\rightarrow \nZ$$
to push this out to a $\nZ$-torsor. That is, define
$$L(\rho_S):=\Sigma_*[ H(\rho_S)  ].$$
The natural map $H(\rho_S)\rightarrow L(\rho_S)$ will also be denoted by the sum symbol $\Sigma$.
\ms

In fact, $L$ extends to a functor from $\mathcal{M}_S^{loc}(A)$ to the category of $\nZ$-torsors.
To carry out this extension,  we just need to extend $H$ to a functor to $H^2_S$-torsors.
 According to Appendices A and B, for $a=(a_v)_{v\in S}\in A^S$ and each $v$, there is an element
$h_{a_v}\in C^2(A, \Z/n)/B^2(A, \Z/n)$ such that
$$c\circ Ad_{a_v}=c+ dh_{a_v}.$$
Also, 
$$h_{a_vb_v}=h_{a_v}\circ \Ad_{b_v}+ h_{b_v}.$$
Hence, given $a: \rho_S\rightarrow \rho_S',$  so that $\rho_S'=\Ad_a \circ \rho_S$, we define
$$H(a): H(\rho_S)\rightarrow H(\rho_S')$$
to be the map induced by
$$x\mapsto x'=x+(h_{a_v}\circ \rho_v)_{v\in S}.$$
Then
$$dx'=dx+(d(h_{a_v}\circ \rho_v))_{v\in S}=(c\circ \rho_v)_{v\in S}+((dh_{a_v})\circ \rho_v)_{v\in S}=(c\circ \Ad_{a_v}\circ \rho_v)_{v\in S}.$$
So $$x'\in d^{-1}(c\circ \rho'_S)/B^2_S,$$ and by the formula above, it is clear that $H$ is a functor. That is, $ab$ will send $x$ to
$$x+h_{ab}\circ \rho_S,$$
while if we apply $b$ first, we get $$x+h_b\circ \rho_S\in H(\Ad_b\circ \rho_S),$$ which then goes via $a$ to
$$x+h_b\circ \rho_S +h_a\circ \Ad_b\circ \rho_S.$$
Thus, $$H(ab)=H(a)H(b).$$ Defining
$$L(a)=\Sigma_*\circ H(a)$$
turns $L$ into a functor from $\mathcal{M}_S^{loc}$ to $\nZ$-torsors. Even though we are not explicitly laying down geometric foundations, it is clear that $L$ defines thereby an $A^S$-equivariant $\nZ$-torsor on $Y_S^{loc}(A)$, or a $\nZ$-torsor on the stack $\msloc(A)$.
\ms

We can compose the functor $L$ with the restriction
$r_S: \mathcal{M}_S(A)\rightarrow \mathcal{M}_S^{loc}(A)$ to get an $A$-equivariant functor
$L^{glob}$ from $\mathcal{M}_S(A)$ to $\nZ$-torsors. 

\begin{lem}
Let $\rho \in Y_S(A)$ and $a\in \Aut(\rho).$ Then $L^{glob}(a)=0$.
\end{lem}

\begin{proof}
By assumption, $\Ad_a\rho=\rho$, and hence, $dh_a\circ \rho=0$. That is,
$h_a\circ \rho\in H^2(\pi_S, \nZ).$  Hence, by the reciprocity law for $H^2(\pi_S, \nZ)$  (\cite{NSW}, Theorem (8.1.17)), we get
$$\Sigma_*(h_a\circ \rho )=0.$$
\end{proof}
By the argument of  \cite{FQ}, page 439, we see that there is a $\nZ$-torsor
$$L^{\inv}([\rho])$$ of invariant sections for the functor $L^{glob}$ depending only on the orbit $[\rho]$. This is the set of families of elements $$x_{\rho'}\in L^{glob}(\rho')$$ as $\rho'$ runs over $ [\rho]$ with the property  that every morphism $a: \rho_1\rightarrow \rho_2$ takes $x_{\rho_1}$ to $x_{\rho_2}$.  Alternatively, $L^{\inv}([\rho])$ is the inverse limit of the $L^{glob}(\rho')$ with respect to the indexing category $[\rho]$. 

\ms

Since $$H^3(X_S, \nZ)=0,$$(\cite{NSW}, Proposition (8.3.18)) the cocycle
$c\circ \rho$ is a coboundary $$c\circ \rho=d\b$$ for
$\b\in C^2(\pi_S, \nZ)$. This element defines a class $$CS_c([\rho]):=\Sigma([i^*_S(\b)])\in L^{\inv}([\rho]).$$
A different  choice $\b'$ will be related by 
$$\b'=\b+z$$
for a 2-cocycle $z\in Z^2(\pi_S, \nZ)$, which vanishes when mapped to $L((\rho\circ i_v)_{v\in S})$. Thus, the class $CS_c([\rho])$ is independent of the choice of $\b$ and defines a global section
$$CS_c\in \G( \mathcal{M}_S(A), L^{glob}).$$
Within the context of this paper, a `global section' should just be interpreted as an assignment of $CS_c([\rho])$ as above for each orbit $[\rho]$.

\section{The arithmetic Chern-Simons action: the $p$-adic case} \label{sec:p-adic case}

Now fix a prime $p$ and assume all primes of $F$ dividing $p$ are contained in $S$.   Fix a compatible system $(\zeta_{p^n})_n$ of $p$-power roots of unity, giving us an isomorphism 
$$
\z: \Z_p\simeq \Z_p(1):=\invlim_n \mu_{p^n}.
$$
In this section, we will be somewhat more careful with this isomorphism. Also, it will be necessary to make some assumptions on the representations that are allowed. 
\ms

Let $A$ be a $p$-adic Lie group, e.g., $GL_n(\Z_p)$. Assume $A$ is equipped with an open  homomorphism
$t:A \rightarrow \G:=\Z_p^\xx$ and define $A^n$ to be the kernel of the composite map 
$$
A\rightarrow \Z_p^\xx\rightarrow (\Z/p^n\Z)^\xx=:\G_n.
$$
Let 
$$
\Ainfty=\cap_n A^n=\Ker(t).
$$ 

In this section, we denote by $Y_S(A)$ the continuous homomorphisms $$\rho: \pi_S\rightarrow A$$ such that $t\circ \rho$ is a  power $\chi^s$ of the $p$-adic cyclotomic character of $\pi_S$ by a $p$-adic unit $s$. (We note that $s$ itself is allowed to vary.) Of course this condition will be satisfied by any geometric Galois representation or natural $p$-adic families containing one.

\ms
 
    As before, $A$ acts on $Y_S(A)$ by conjugation. But in this section, we will restrict the action to
$\Ainfty$ and   use the notation $\mathcal{M}_S(A)$ for the corresponding action groupoid.

Similarly, we denote by $Y_S^{loc}$ the collections of continuous homomorphisms
$\rho_S=(\rho_v: \pi_v\ra A)_{v\in S}$
for which there exists a $p$-adic unit $s$ such that $t\circ \rho_v=(\chi|\pi_v)^s$ for all $v$. $\mathcal{M}^{loc}_S(A)$ then denotes the action groupoid defined by the product   $(\Ainfty)^S$ of the conjugation action on  the $\rho_S$.
\ms

We now fix a continuous cohomology class $$c\in  H^3(A, \Z_p[[\G]]),$$where 
$$\Z_p[[\G]]=\invlim_n\Z_p[\G_n].$$ We represent $c$  by a cocycle in $Z^3(A, \Z_p[[\G]])$, which we will also  denote by $c$.
Given $\rho\in Y_S(A)$, we can view $\Z_p[[\G]]$ as a continuous representation of $\pi_S$, where the action is left multiplication via $t\circ \rho$. We denote this representation by
$\Z_p[[\G]]_{\rho}$. The isomorphism
 $\z:\Z_p\simeq \Z_p(1)$, even though it's not $\pi_S$-equivariant, does induce a $\pi_S$-equivariant isomorphism
$$
\z_{\rho}: \Z_p[[\G]]_{\rho}\simeq \L:=\Z_p[[\G]]\otimes \Z_p(1).
$$
Here, $\Z_p[[\G]]$ written without the subscript refers to the action via the cyclotomic character of $\pi_S$ (with $s=1$ in the earlier notation). The isomorphism is defined as follows.
If $t\circ \rho=\chi^s$, then we have the isomorphism
$$\Z_p[[\G]]\simeq \Z_p[[\G]]_{\rho}$$
that sends $\g$ to $\g^s$. On the other hand, we also have
$$\Z_p[[\G]]\simeq \L$$
that sends $\g$ to $\g\otimes \g\z(1).$ Thus, $\z_{\rho}$ can be taken as the inverse of the first followed by the second.

  \ms
Combining these considerations,  we get an element
$$
\z_{\rho}\circ \rho^*c=\zeta_{\rho}\circ c\circ \rho \in Z^3(\pi_S, \L). 
$$
Similarly, if $\rho_S:=(\rho_v)_{v\in S}\in Y^{loc}_S$, we can regard $\Z_p[[\G]]_{\rho_v}$ as a representation of $\pi_v$ for each $v$, and  we get
$\pi_v$ equivariant isomorphisms
$$
\z_{\rho_v}:\Z_p[[\G]]_{\rho_v}\simeq \L.
$$
We also use the notation $$\z_{\rho_S}:\prod_{v\in S}\Z_p[[\G]]_{\rho_v}\simeq \prod_{v\in S}\L$$
for the isomorphism given by the product of the $\z_{\rho_v}$.

\ms

It will be convenient to again denote by  $C^i_S(\L)$ the product $\prod_{v\in S} C^i(\pi_v, \L)$ and use the similar notations $Z^i_S(\L)$, $B^i_S(\L)$ and $H^i_S(\L)$.
The element $\z_{\rho_S}\circ \rho_S^*c$ is  an element in $ Z^3_S(\L)$.
We then put
$$
H(\rho_S, \L):=d^{-1}((\z_{\rho_S}\circ \rho_S^*c) )/B^2_S(\L)\subset C^2_S(\L)/B^2_S(\L).
$$
This is a torsor for $$H^2_S(\L)\simeq \prod_{v\in S} H^2(\pi_v, \L).$$
The augmentation map $$a:\L\rightarrow \Z_p(1)$$ for each $v$ can be used to push this out to a torsor
$$a_*(H(\rho_S, \L))$$
for the group 
$$\prod_{v\in S} H^2(\pi_v, \Z_p(1))\simeq \prod_{v\in S}\Z_p,$$  which then can be pushed out with the sum map $$\Sigma :\prod_{v\in S}\Z_p\rightarrow \Z_p$$ to give us a $\Z_p$-torsor $$L(\rho_S, \Z_p):=\Sigma_*(a_*(H(\rho_S, \L))).$$

As before, we can turn this into a functor
$L(\cdot, \Z_p)$ on $\mathcal{M}^{loc}_S(A)$, taking into account the action of $(\Ainfty)^S$. By composing with the restriction functor
$$
r_S:\mathcal{M}_S(A)\ra \mathcal{M}^{loc}_S(A),
$$
we also get a $\Z_p$-torsor $L^{glob}(\cdot, \Z_p)$
on $\mathcal{M}_S(A)$.
\ms

We now  choose an element $\b\in C^2(\pi_S, \L)$ such that
$$
d\b=\z_{\rho}\circ c\circ \rho\in Z^3(\pi_S, \L)=B^3(\pi_S, \L)
$$ 
to define the $p$-adic Chern-Simons action
$$
CS_c([\rho]):=\Sigma_*a_*i_S^*(\b)\in L^{glob}([\rho],\Z_p).
$$
The argument that this action is independent of $\b$ and equivariant is also the same as before, giving us an element 
$$
CS_c\in \G(  \mathcal{M}_S(A) ,L^{glob}(\cdot, \Z_p)).
$$

\section{Remarks}
1. The restrictions (1) and (2) on the  representations $\rho$ that make up $Y_S(A)$ in section 3   might seem rather stringent. However, if we take $A$ to be the image of some fixed $p$-adic geometric Galois representation $\rho_0$, this includes all twists $\rho_0(s)$ of $\rho_0$ by unit powers $\chi^s$ of the $p$-adic cyclotomic character. Thus, we are in effect constructing with the cocycle $c$ a  section of a line bundle on the  entire $p$-adic weight space $\Z_p^\xx$. In the next section, we will discuss the motivation coming from the theory of $L$-functions. The ability to construct such a section is already promising from this point of view.
\ms

2. We have dealt with the $p$-adic theory assuming $S$ is non-empty. It is straightforward to get a $p$-adic function  on the moduli space for $X$, the case `without boundary'. But according to the Fontaine-Mazur conjecture, an infinite $p$-adic Lie group should not be possible as the image of a representation of $\pi_1(X,b)$. Indeed, since $CS_c(\rho)$ is a $p$-adic invariant of such a represention, plausible applications to questions of  existence and distribution could be considered.
\ms

3. In the $p$-adic theory, no changes are necessary for $F$ with a real embedding provided we take $p\neq 2$. Indeed, even though the duality theorems involving the sheaf $\Gm$ become somewhat more complicated because of the contribution from real places, such contributions all vanish for $p$-adic coefficient sheaves if $p$ is odd. However, if one were to imagine a Chern-Simons theory for complex $L$-functions, the Archimedean places should be expected to play an essential role.
\ms

4. In the first two sections, we assumed the field $F$ contained the $n$-th roots of 1 so as to trivialize the sheaf $\mu_n$. This allowed us to construct functions out of constant cohomology classes for $A$. Similarly, in section 3, we obtained $\Z_p(1)$  cohomology classes from $\Z_p$-classes by a twisting trick familiar in Iwasawa theory. To avoid this, one could have regarded the group $A$ as a constant sheaf and used cohomology classes in $H^3(BA, \mu_n)$ from the beginning. But it is hard to imagine constructing such classes other than by twisting classes with constant coefficients. This is essentially equivalent to our approach. 
\ms

5. We are not giving at present any examples. For finite groups $A$, it is not hard to get classes in $H^3$, for example, starting from cyclic subgroups. On the other hand, a norm compatible sequence of classes for infinite $p$-adic Lie groups seems to be harder to construct. In  subsequent work, we will study this question systematically from the viewpoint of Lazard's theory of analytic groups and duality for groups like $GL_n(\Z_p)$ \cite{huber}.
\ms

6. It is unfortunate that the $p$-adic case does not include $A=\Z_p$ for reasons of cohomological dimension. Even in topological Chern-Simons theory, the abelian case seems to have a nature different from groups like $SU(2)$. One way of getting around this difficuly for  $A\simeq \Z_p^r$ might be to use classes in $H^1(A,  \Z_p)$ pulled back to $\pi_S$, from which one could take Massey products to end up with 3-cocycles. Another possibility, following a pattern familiar in Iwasawa theory, would be to find a sequence of $\Z/p^n\Z$ classes that are congruent in a somewhat  subtle sense, to which one  applies the construction at the end of section 1.
\ms

7. One notable difference from the usual Chern-Simons theory is that the  Chern-Simons line of this paper is presented as an additive torsor, rather than a multiplicative one.  However, note that we are using  an isomorphism $\nZ\simeq \mu_n$, and the latter is multiplicative. Thus, our finite torsors can also be thought of as multiplicative $\mu_n$-torsors, in closer parallel to the topological setting.

However, the $p$-adic Chern-Simons line does seem to be genuinely additive. As will be explained in the next section, the values of $p$-adic $L$-functions should also lie in the fibers of a line bundle. Thus, if there is a connection between the two, the arithmetic Chern-Simons invariant should be related to the {\em logarithm} of the $p$-adic $L$-function.
\ms

8. In this paper, we are defining only the classical Chern-Simons functional. Speculating wildly, one might hope that twists of the value of a  classical functional by a family of cyclotomic characters represent a kind of semi-classical approximation. In any case, it would be interesting to construct a quantum wavefunction in the arithmetic setting. For the finite-coefficient case of sections 1 and 2, this is in principle easy to define. The (more important) $p$-adic coefficients present a greater challenge.
\ms

9. Since the $\Spec(F_v)$ are playing the role of boundary tori,  moduli spaces of local Galois representations should make up the classical phase spaces of arithmetic Chern-Simons theory. In the topological case, the corresponding moduli space has an interpretation using either holomorphic vector bundles or Higgs bundles, depending on the group. In this regard, it is interesting to take note of recent developments in $p$-adic Hodge theory defining a functor from Galois representations to vector bundles on a $p$-adic curve \cite{FF}. The moduli space of vector bundles that arises admits a uniformization by an infnite-dimensional Grassmannian in essentially the same manner as for complex Riemann surfaces. The possibility of using this construction to study determinant line bundles following the pattern of conformal field theory appears to be an interesting avenue of investigation in the study of local moduli spaces.
\ms

10. It is somewhat unforunate in this regard that work of  Kapustin and Witten \cite{KW} on the geometric Langlands programme doesn't make use of Chern-Simons theory, but rather,  $S$-duality for 4D gauge theory. Since the Langlands programme is another source of $L$-functions in arithmetic, a pleasant coincidence might have been for  topological Chern-Simons theory to play a critical role also in the geometric Langlands programme. In any case, the analogy between  Chern-Simons functions and $L$-functions suggests a possibility for defining $L$-functions in geometric Langlands, usually thought not to admit such a formalism. That is, the $L$-function on the geometric Galois side should have the structure of a wavefunction over a character variety. The role of automorphic forms in geometric Langlands is played by $D$-modules on moduli spaces of principal bundles that are Hecke eigensheaves in a suitable sense. The theory of automorphic $L$-functions should then assign an amplitude to such a $D$-module, possibly using a path integral over objects on a three manifold that have the given $D$-module as a boundary value.

\section{Towards computation}
In this section, we indicate how one might go about computing the Chern-Simons invariant in the unramified case with finite coefficients. That is, we assume we are in the setting of section 1.
\ms

Let $X=\Spec(\O_F)$ and $M$ a continuous representation of $\pi=\pi_1(X)$ regarded as a locally constant sheaf on $X$. Assume $M=\invlim M_n$ with  $M_n$ finite representations such that there is a finite set $T$  of primes in $\O_F$ containing all primes dividing the order of any $|M_n|$. Let $U=\Spec(\O_{F,T})$, $G_T=\pi_1(U)$, and $G_v=\Gal(\bF_v/F_v)$ for a place $v$ of $F$. Write $m_v$ for the maximal ideal of $\O_F$ correponding to the place $v$ and $r_v$ for the restriction map of cochains or cohomology classes from $G_T$ to $G_v$.
\ms

Denote by
$C^*_c(G_T, M)$ the complex defined as a mapping fiber
$$C^*_c(G_T, M):=\mbox{Fiber}[C^*(G_T, M)\ra \prod_{v\in T} C^*(G_v, M)].$$
So
$$C^n_c(G_T, M)=C^n(G_T,M)\times\prod_{v\in T} C^{n-1}(G_v, M),$$
and
$$d(a, (b_v))=(da, (r_v(a)-db_v))$$
for $(a,(b_v))\in C^n_c(G_T, M)$.
As in \cite{FK} page 20, since there are no real places in $F$, there is a quasi-isomorphism
$$C^*_c(G_T, M)\simeq R\G(U, j_!(M)),$$
where  $j:U\ra X$ is the inclusion.
But  there is also an exact sequence
$$0\ra j_!j^*(M)\ra M\ra i_*i^*(M)\ra 0,$$
where $i:T\ra X$ is the closed immersion complementary to $j$. Thus, we get an exact sequence
$$\prod_{v\in T}H^2(\Spec(\O_F/m_v), i^*(M))\ra H^3(C_c(G_T, M))\ra H^3(X,M)\ra \prod_{v\in T}H^3(\Spec(\O_F/m_v),$$
from which we get an isomorphism
$$H^3(C_c(G_T, M))\simeq H^3(X,M),$$
since $\Spec(\O_F/m_v)$ has cohomological dimension 1.

We interpret this as a statement that the cohomology of $X$
$$H^3(X, M)$$
can be identified with  cohomology of a `compactification' of $U$ with respect to the `boundary',  that is, the union of the $\Spec(F_v)$ for $v\in T$. 
This means that a class $z\in H^3(X, M)$ is represented by
$(c, (b_v)_{v\in T})$, where $c\in Z^3(G_T, M)$ and
$b_v\in C^2(G_v, M)$ in such a way that 
$$db_v=c|G_v.$$
There is also the exact sequence
$$\rightarrow H^2(G_T, M)\rightarrow \prod_{v\in T} H^2(G_v, M)\rightarrow H^3_c(U, M)\rightarrow 0,$$
the last zero being $H^3(U, M)=0$.
We can use this to compute the invariant of $z$ when $M=\mu_n$.
We have to lift $z$ to a collection of classes $x_v\in H^2(G_v, \mu_n)$ and then take the sum
$$\inv(z)=\sum_v\inv(x_v).$$
This is independent of the choice of the $x_v$ by the reciprocity law. The lifting process may be described as follows. The map $$\prod_{v\in T} H^2(G_v, \mu_n)\rightarrow H^3_c(U, \mu_n)$$
just takes a tuple of 2-cocycles $(x_v)_{v\in T}$ to $(0, (x_v)_{v\in T})$. But by the vanishing of $H^3(U, \mu_n)$, given
$z=(c, (b_v))$, we can find a global cochain $a\in C^2(G_T, \mu_n)$ such that $da=c$. We then put
$x_v:= b_v-r_v(a).$

\ms

When we start with a class $z\in H^3(\pi, \mu_n)$ let
$c\in Z^3(\pi, \mu_n)$ represent $z$. Let $I_v\in G_v$ be the inertia subgroup. We now can trivialise $c|G_v$ by first trivialising it over $G_v/I_v$ to which it factors. That is,
the $b_v$ as above can be chosen as cochains factoring through $G_v/I_v$.
This is possible because $H^3(G_v/I_v, \mu_n)=0$. The class $(c, (b_v))$ chosen this  way is independent of the choice of the $b_v$. This is because $H^2(G_v/I_v, \mu_n)$ is also zero.
The point is that the representation of $z$ as $(c, (b_v))$ with unramified $b_v$ is essentially canonical. More precisely, given $c|(G_v/I_v)\in Z^3(G_v/I_v, \mu_n)$, there is a canonical $$b_v\in C^2(G_v/I_v, \mu_n)/B^2(G_v/I_v, \mu_n)$$
such that $db_v=c|(G_v/I_v)$. This can then be lifted to a canonical class
in $ C^2(G_v, \mu_n)/B^2(G_v, \mu_n)$.
 Now we trivialise $c|G_T$ globally as above,
that is, by the choice of $a\in C^2(G_T, \mu_n)$ such that
$da=c|G_T$. Then
$((b_v-r_v(a))_{v\in T}$ will be cocycles, and we compute
$$\inv(z)=\sum_v \inv (b_v-r_v(a)).$$
\ms

A few remarks about this method:
\ms

1. Underlying this is the fact that the the compact support cohomology $H^3(U, \mu_n)$ can be computed relative to the somewhat fictitious boundary of $U$ or as relative cohomology $H^3(X, T; \mu_n).$ Choosing the unramified local trivialisations corresponds to this latter representation.
\ms

2. To summarise the main idea again, starting from a cocycle $c\in Z^3(\pi, \mu_n)$ we have canonical unramified trivalisations at each $v$ and a non-canonical  global ramified trivialisation. \bq {\em The invariant of $z$ measures the discrepancy between the unramified local trivialisations and a ramified global trivialisation. }\eq
The fact that the non-canonicality of the global trivialisation is unimportant follows from the reciprocity law. 
\ms

3. The description above that computes the invariant by comparing the local unramified trivialisation with the global ramified one is a precise analogue of the so-called `glueing formula' for Chern-Simons invariants when applied to $\rho^*(c)$ for a representation $\rho: \pi\ra \nZ$ and a 3-cocycle $c$ on $\nZ$. A systematic treatment with explicit examples will be presented in the forthcoming work \cite{CKPY}.
\ms

For the moment, we content ourselves with some ideas for   the case of
$\Hom(\pi, \Z/p)$.
\ms

Recall from section 1 that a 3-cocycle on $\Z/p$ can be obtained as $\d \a\cup \a$, where $\a\in H^1(\Z/p, \Z/p)$ is the identity map
and $\d$ is the boundary map coming from the extension
$$E: 0\rightarrow \Z/p\rightarrow \Z/p^2\rightarrow \Z/p\rightarrow 0.$$
If we have a homomorphism
$$f: N\rightarrow \Z/p,$$
a trivialisation of $f^*(\d \a\cup  \a)$ may be obtained by trivialising $\d\a$. That is, if $db=f^*(\d \a)$, for a cochain
$b $ on $N$, then
$$d(-\a \cup b)=\a\cup \d \a.$$
Another way of putting this is that a splitting of the sequence
$f^*(E)$ will give a trivialisation. That is, if there is a lifting
$\tf: N\rightarrow \Z/p^2$ of $f$, then we can construct a trivialisation. An explicit description goes like this. Choose a set-theoretic splitting
$s: \Z/p\rightarrow \Z/p^2,$ for example, in the standard way that sends the class of $i \mod p$ to that of $i\mod p^2$.
Then 
$\d \a= ds.$  Suppose
$\tf$ exists as above. Then the trivialisation of $f^*\d \a$ is given by 
$$b:=s\circ f -\tf,$$
so that $-\a \cup (s\circ f-\tf)$ is a trivialisation of
$\a \cup \d \a $. Now, if $N=G_v/I_v\simeq \hZ$, it suffices to choose $\tf$ in any manner. So the key point is the lifting 
$\tf$ in the case where $N=G_T$ and $f:G_T\rightarrow \Z/p$ is the composition of a representation $\rho:\pi\rightarrow \Z/p$ with the quotient map $k: G_T\rightarrow \pi$.
 To construct examples, here is a simple starting point. Take $F$ an  totally imaginary field such that the class group $C_F\simeq \Z/p$.  I believe there are many examples where the Hilbert class field of $F$ has been constructed as a Kummer extension, even though we need to look through the literature on explicit class field theory (say with $F=\Q(\mu_{p^2}))$. Let $H=F(h^{1/p})$ and let $\rho: \pi\rightarrow \Z/p$ be the corresponding Kummer character. With these assumptions, of course there can't be a lift $\tilde{\rho}:\pi\rightarrow \Z/p^2$. However, by taking $T$ to be the ramified places of the character correponding to $ h^{1/p^2}$, $f:=\rho\circ k$ does  lift to $\tf:G_T\rightarrow \Z/p^2$. This then gives the trivialisation of $f^*(\d \a)$ as above.

\section{Motivation: $L$-functions}
In the following, the ring $R$ can be provisionally thought of as either $\C$, $\Z_p$, or $\Q_p$ for some primes $p$. However, one can, and needs to, allow more general coefficients, such as an extension field of $\Q_p$, or the profinite group rings of  Iwasawa theory (\cite{FK}, 1.4.1). It is conceivable that more general rings are appropriate for the complex theory as well. However, for concreteness, it is all right to  keep in mind these simple cases.
\ms

The theory of $L$-functions, still largely conjectural, assigns a canonical {\em $L$-amplitude}
$$L(X,\mathcal{F})$$
to a pair 
 consisting of a scheme $X$ of finite type over $\Z$ and a constructible sheaf $\mathcal{F}$ of finitely-generated $R$-modules in the \'etale topology of $X$. It is convenient to allow also elements of bounded derived categories of such $\mathcal{F}$ as coefficients. This amplitude is sometimes a number in $R$, but  is expected in general to be an element of a determinant line. The proposal that an amplitude of the right sort can always be defined is known as the Hasse-Weil conjecture for complex $L$-functions and Iwasawa's main conjecture for $p$-adic $L$-functions. The main difficulty can be thought of as a problem of regularizing an infinite product.  Since this point of view may not be entirely familiar to physicists, we give a brief overview of the theory described in \cite{kato} and \cite{FK}.
 
 \ms
 
 Associated to $(X,\mathcal{F})$, there are the cohomology groups with compact support
 $$H^i_c(X, \mathcal{F}),$$
 which are finitely generated $R$-modules. We denote by $D(X, \mathcal{F})$ the dual of the determinant of cohomology
 $$D(X, \mathcal{F}):=\otimes_i \det H^i(X, \mathcal{F})^{(-1)^{i+1}},$$
 a projective $R$-module of rank 1 \cite{KM}. 
 Hence, if $\mathcal{M}$ is a moduli space of sheaves on $X$,  the $D(X, \mathcal{F})$ will vary over points $[\mathcal{F}]\in \mathcal{M}$ and come together to form a line bundle\footnote{For this motivational discussion, the precise conditions necessary for the geometric statement to hold will be left unstated.}
 $$\mathcal{D}\rightarrow \mathcal{M}.$$
 Note here that $\mathcal{M}$ will be like the representation varieties in complex geometry, and hence, have the structure of  a scheme, formal scheme, or an analytic space over $\Spec(R)$.
 \ms
 
 The $L$-amplitude is conjectured to be a generator
 $$L(X, \mathcal{F})\in D(X,\mathcal{F}),$$
 which should patch together to a trivialisation of $\mathcal{D}$ over $\mathcal{M}$. Thus, the theory of $L$-functions proposes the existence of a canonical section
 $$L(X, \cdot)\in \G(\mathcal{M}, \mathcal{D})$$
 for suitable moduli spaces $\mathcal{M}$ of sheaves. The techniques of arithmetic geometry have so far provided essentially ad hoc methods for constructing such sections in limited settings. Thus, the availability of  solutions to entirely analogous problems in quantum field theory is the main motivation for an attempt to develop a parallel arithmetic theory.
 \ms

 A sheaf $\mathcal{F}$ is {\em acyclic} if $H^i_c(X, \mathcal{F})=0$ for all $i$. For an acyclic sheaf $\mathcal{F}$, there is a canonical trivialisation
 $$D(X, \mathcal{F})\simeq R$$
 corresponding to the fact that the determinant of the zero module is $R$. For acyclic sheaves, the $L$-amplitude can be regarded as an element of $R$. Furthermore,
 over the locus $\mathcal{M}_{acyc}\subset \mathcal{M}$ of acylic sheaves, we expect the determinant line bundle to have a canonical trivialization
 $$\mathcal{D}|\mathcal{M}_{acyc}\simeq \mathcal{O}_{\mathcal{M}_{acyc}}.$$
 Thus, over $\mathcal{M}_{acyc}$, the $L$-amplitude can be regarded as a function.
 \ms
 
 For coefficient rings like $R=\Z_p$, even when $\mathcal{F}$ is not acyclic, $\mathcal{F}\otimes \Q_p$ may be acyclic. So even when an element in $D(X, \mathcal{F})$ may not  canonically be an element of $R$, it may sometimes be regarded as an element of $R\otimes \Q_p$.
 A related phenomenon is the following. Suppose $$\mathcal{M}=\Spec(T)$$ and the locus of non-acyclic sheaves form a divisor with equation $f=0$. Then
 $\mathcal{D}$ can be regarded as a $T$-module. And $$\mathcal{D}[1/f]=\mathcal{D}\otimes T[1/f]$$ is canonically trivial. Let $s$ be the section of $\mathcal{D}[1/f]$ corresponding to 1 under this trivialization. Then, in favorable circumstances, for example, if $\mathcal{M}$ is regular, the section $$(1/f)s$$ extends over all of $\mathcal{M}$ and can be regarded as a trivializing section of $\mathcal{D}$. This is the way in which characteristic elements that occur in classical formulations of the Iwasawa main conjecture become interpreted as trivializing sections of determinant lines (cf. \cite{FK}, Example 2.5).
 \ms

The $L$-amplitude is conjectured to satisfy some natural conditions (\cite{kato}, conjecture 3.2.2, modified by \cite{FK}, conjecture 2.3.2):
\ms

 (1) Multiplicativity:
 If
 $$0\rightarrow \mathcal{F}_1\rightarrow \mathcal{F}_2\rightarrow \mathcal{F}_3\rightarrow 0$$
is a exact sequence, then the canonical isomorphism
 $$D(X, \mathcal{F}_2)\simeq D(X, \mathcal{F}_2)\otimes D(X, \mathcal{F}_2)$$
 takes
 $L(X. \mathcal{F}_2)$ to $L(X. \mathcal{F}_1)\otimes L(X. \mathcal{F}_3)$.
 \ms

 (2) Compatibility change of coefficient rings: If $R'$ is an $R$-algebra and $\mathcal{F}'=\mathcal{F}\otimes^L R'$, then the natural isomorphism 
 $$D(X, \mathcal{F})\otimes_R R' \simeq D(X, \mathcal{F}')$$
 takes $L(X, \mathcal{F})\otimes 1$ to $L(X, \mathcal{F}')$. (The base-change considered in \cite{FK} is more general to accommodate the possibility of non-commutative coefficient rings.)
 \ms
 
 (3) Two normalisation conditions: an easy one for sheaves over a finite field, and a very hard one having to do with conjectures on $L$-amplitude of motives.
 \ms
 
 We comment on (1) and (3).  The most important case of (1) is
 $$0\rightarrow j_!(j^{-1}\mathcal{F})\rightarrow \mathcal{F} \rightarrow i_*(i^{-1}(\mathcal{F}))\rightarrow 0,$$
 where $i: Z\rightarrowtail X$ is a closed embedding and $j:U\rightarrowtail X$ is the complement.
 Then the required multiplicativity is
 $$L(X, \mathcal{F})=L(U, \mathcal{F})\otimes L(Z, \mathcal{F}),$$
 where we omit the inverse images for notational convenience. Note that when all three are acyclic, the tensor product becomes a product of numbers and this is a literal equality.
 \ms
 
 The easy normalisation condition in (3)  is when $X=\Spec (\F_q)$, the spectrum of a finite field with $q=p^d$ elements. In that case, the stalk $\mathcal{F}_x$ at a geometric point
 $$x:\Spec(\bar{\mathbb{F}}_q)\rightarrow \Spec(\F_q)$$
 carries an action of the geometric Frobenius $$Fr_x: \Spec(\bar{\mathbb{F}}_q)\rightarrow \Spec(\bar{\mathbb{F}}_q)$$
 (the dual to the map $a\mapsto a^{q^{-1}}$). Thus, we get an exact sequence
 $$0\rightarrow H^0(\mathcal{F})\rightarrow \mathcal{F}_x\stackrel{I-Fr_x}{\longrightarrow} \mathcal{F}_x\rightarrow H^1(\mathcal{F})\rightarrow 0,$$
 inducing an isomorphism
 $$D(\Spec(\F_q), \mathcal{F})\simeq \det (\mathcal{F}_x)^*\otimes \det(\mathcal{F}_x) \simeq R.$$
 Then $L(\Spec(\F_q), \mathcal{F})$ is defined to be the inverse image of 1. When $\mathcal{F}_x$ is $R$-free and $\mathcal{F}$ is acyclic,
 this gives the normalization
 $$L(\Spec(\F_q), \mathcal{F})=\frac{1}{\det( [I-Fr_x]|\mathcal{F}_x)}.$$
When $X=\Spec(\F_q)$, the category of sheaves of $R$-modules is equivalent to the category of continuous representations of $\Gal(\bar{\mathbb{F}}_q/\F_q)$ on $R$-modules. This Galois group is topologically generated by $Fr_x$. There are  a number of ways of setting up the formalism of sheaves so that arbitrary representations of the Weil group $W_{\F_q}\subset \Gal(\bar{\mathbb{F}}_q/\F_q)$, that is, the group of integer powers of $Fr_x$, define sheaves on schemes over $\F_q$. Since $W_{\F_q}\simeq \Z$,  the one-dimensional complex characters of the Weil group of $\Spec(\F_q)$ are parametized by $\C^\xx$. So they can all be written as
 $$Fr_x\mapsto q^{-s},$$
 for some $s\in \C$. (The reason we parametrize the characters this way is because it is the description that's compatible with the norm character on the global idele class group.)  We denote the 1-dim representation corresponding to this character $\C(s)$. When $\mathcal{F}$ is a sheaf of $\C$-vector spaces, we denote by $\mathcal{F}(s)$ the sheaf corresponding to the representation
 $\mathcal{F}_x\otimes \C(s)$. 
  If $\mathcal{F}(s)$ is acyclic, we get
 $$L(\Spec(\F_q), \mathcal{F}(s))=\frac{1}{\det( [I-p^{-s}Fr_x]|\mathcal{F}_x)}.$$
 This is the way in which the analytic $L$-factors that arise in the complex theory of $L$-functions come up naturally as we vary a representation in a canonical one-parameter family.
 \ms
 
 For general $X$, let $S$ be a finite subset of $X_0$, the set of closed points of $X$, and $U_S=X\setminus S$.
 Then the multiplicative property of the $L$-amplitude gives
$$L(X, \mathcal{F})=L(U_S, \mathcal{F})\prod_{y\in S} L(\Spec(k(y)), \mathcal{F}_y),$$
where $k(y)$ is the (finite) residue field at $y$.
If the limit as $S$ grows large exists, we should have
$$L(X, \mathcal{F})=L(\mbox{generic}, \mathcal{F})\prod_{y\in X_0} L(\Spec (k(y)), \mathcal{F}_y),$$
where the factor $L(\mbox{generic}, \mathcal{F})$ can sometimes be determined. In substantial generality, it can be shown that the limit exists when we replace $\mathcal{F}$ by $\mathcal{F}(s)$ for $Re(s)$ sufficiently large, forcing on us essentially the familiar definition of an $L$-amplitude as an infinite product. There is also a formalism for making sense of this for coefficient rings more general than $\C$ (subject to hard conjectures and theorems about Weil sheaves associated to $l$-adic sheaves). The usual Hasse-Weil conjecture asserts that when $\mathcal{F}$ is motivic, one can define $L(X, \mathcal{F}(s))$  in a way that's meromorphic in $s$, with poles contributed only by trivial sheaves.
\ms

The hard (and important) normalisation condition would require lengthy prerequisites, and will not be discussed here at all.  The reader is referred  to \cite{kato, FK}.
\ms

Now we specialise to the situation where $X=\Spec(\mathcal{O}_F)$ as in the earlier sections, and $X_S=\Spec(\mathcal{O}_F[1/S])$ for  a finite set of primes $S$. 
As indicated above, a $p$-adic $L$-function  is supposed to be a section of
$\mathcal{D}$ on $\mathcal{M}_S$:
$$L(X, \cdot)\in \G(\mathcal{M}_S, \mathcal{D}).$$
In this paper, we have constructed in section 3 
$$CS_c(\cdot)$$an additive version of such a section, at least for a restricted family.
The optimistic wish referred to in the abstract is a comparison
$$CS_c(\cdot) \sim \log L(X, \cdot).$$
To effect such a comparison, one would obviously have to relate the $\Z_p$-torsors constructed in an elementary fashion to the determinant line bundles. I am told by Dan Freed that such a comparison is not available even in topological Chern-Simons theory, and may be rather difficult. Nevertheless, the strong analogy between the multiplicativity of $L$-functions and the glueing formula seems worth investigating in detail.
\ms

Bruce Bartlett has emphasised to me the importance of Reidemeister torsion within this circle of ideas. Indeed, Witten \cite{witten} had already noted that the square root of Reidemeister torsion appears as the main contribution to the semi-classical Chern-Simons wavefunction by a classical minimum. Since there has been for some time a folklore analogy in number theory between $L$-functions and Reidemeister torsion (cf. \cite{deninger}), a reasonable avenue of investigation might be a definition of an arithmetic Reidemeister torsion using the arithmetic Chern-Simons functional, which could then be compared to the $L$-amplitude.
\ms

The main point is important enough to be worth repeating: it is a major unsolved problem of arithmetic geometry to define global sections of determinant line bundles satisfying the natural properties outlined above. The speculations of this section were motivated by the wishful thought that ideas from physics could be employed to effect such a definition. The constructions of the first three sections can be regarded as small beginning steps in this direction.

\section{Appendix A. Conjugation on group cochains}

We compute cohomology of a topological group $G$ with coefficients in a topological abelian group $M$ with continuous $G$-action using the complex whose component of degree $i$ is
$C^i(G, M)$, the continuous maps from $G^i$ to $M$. The differential
$$d: C^i(G,M)\rightarrow C^{i+1}(G,M)$$
is given by
$$df(g_1, g_2, \ldots, g_{i+1})$$
$$=g_1f(g_2, \ldots, g_{i+1})+
\sum_{k=1}^i f(g_1, \ldots, g_{k-1}, g_k g_{k+1}, g_{k+2}, \ldots, g_{i+1})+
(-1)^{i+1} f(g_1, g_2, \ldots, g_i).$$
We denote by
$$B^i(G,M)\subset Z^i(G,M)\subset C^i(G,M)$$
the images and the kernels of the differentials, the coboundaries and the cocycles, respectively. The cohomology is then defined as
$$H^i(G,M):=Z^i(G,M)/B^i(G,M).$$
There is a natural right action of $G$ on the cochains given by 
$$a:  c\mapsto c^a:=a^{-1}c\circ Ad_a,$$
where $Ad_a$ refers to the conjugation action of $a$ on $G^i$. 
\begin{lem} The $G$ action on cochain commutes with $d$:
$$d(c^a)=(dc^a)$$
for all $a\in G$.
\end{lem}

\begin{proof}
If $c\in C^i(G,M)$, then
$$d(c^a)(g_1, g_2, \ldots, g_{i+1})=g_1a^{-1}c(Ad_a(g_2), \ldots, Ad_a(g_{i+1}))$$
$$+
\sum_{k=1}^i a^{-1}c(Ad_a(g_1), \ldots, Ad_a( g_{k-1}), Ad_a( g_k) Ad_a( g_{k+1}), Ad_a(g_{k+2}), \ldots, Ad_a(g_{i+1}))$$
$$+
(-1)^{i+1} a^{-1}c(Ad_a(g_1), Ad_a(g_2), \ldots, Ad_a(g_i))$$
$$=
a^{-1}Ad_a(g_1)c(Ad_a(g_2), \ldots, Ad_a(g_{i+1}))$$
$$+
\sum_{k=1}^i a^{-1}c(Ad_a(g_1), \ldots, Ad_a( g_{k-1}), Ad_a( g_k) Ad_a( g_{k+1}), Ad_a(g_{k+2}), \ldots, Ad_a(g_{i+1}))$$
$$+
(-1)^{i+1} a^{-1}c(Ad_a(g_1), Ad_a(g_2), \ldots, Ad_a(g_i))$$
$$=a^{-1}(dc)(Ad_a(g_1), Ad_a(g_2), \ldots, Ad_a(g_{i+1})$$
$$=(dc)^a(g_1, g_2, \ldots, g_{i+1}).$$
\end{proof}
We use also the notation $(g_1, g_2, \ldots, g_i)^a:=Ad_a(g_1, g_2, \ldots, g_i)$.
It is well known that this action is trivial on cohomology. We wish to show the construction of explicit $h_a$ with the property that $$c^a=c+dh_a$$
for cocycles of degree 1, 2, and 3. The first two are relatively straightforward, but degree 3 is somewhat delicate. In degree 1, first note that
$c(e)=c(ee)=c(e)+ec(e)=c(e)+c(e)$, so that $c(e)=0$. Next, $0=c(e)=c(gg^{-1})=c(g)+gc(g^{-1})$, and hence, $c(g^{-1})=-g^{-1}c(g).$
Therefore, 
$$c(aga^{-1})=c(a)+ac(ga^{-1})=c(a)+ac(g)+agc(a^{-1})=c(a)+ac(g)-aga^{-1}c(a).$$
From this, we get
$$c^a(g)=c(g)+a^{-1}c(a)-ga^{-1}c(a).$$
That is,
$$c^a=c+dh_a$$
for the zero cochain $h_a(g)=a^{-1}c(a).$

\ms

\begin{lem}
For each $c\in Z^i(G,M)$ and $a\in G$, we can associate an $$h^{i-1}_a[c]\in C^{i-1}(G,M)/B^{i-1}(G,M)$$ in such a way that
$$(1)\ \ \ \ \ c^a-c=dh^{i-1}_a[c];$$
$$(2)\ \ \ \ h_{ab}^{i-1}[c]=(h^{i-1}_a[c])^b+h^{i-1}_b[c].$$
\end{lem}
\begin{proof}
This is clear for $i=0$ and we have shown above the construction of $h^0_a[c]$ for $c\in Z^1(G,M)$ satisfying (1). Let us check the condition (2):
$$h^0_{ab}[c](g)=(ab)^{-1}c(ab)$$
$$=b^{-1}a^{-1}(c(a)+ac(b))=b^{-1}h^0_a[c](Ad_b(g))+h^0_b[c](g)=(h^0_a[c])^b(g)+h^0_b[c](g).$$
We prove the statement using induction on $i$, which we now assume to be $\geq 2$.
For a module $M$, we have the exact sequence
$$0\rightarrow M\rightarrow C^1(G,M)\rightarrow N\rightarrow 0,$$
where $C^1(G,M)$ has the right regular action of $G$ and $N=C^1(G,M)/M$. Here, we give $C^1(G,M)$ the topology of pointwise convergence. There is a canonical linear splitting $s: N\rightarrow C^1(G,M)$ with image the group of functions $f$ such that $f(e)=0$, using which we topologise $N$. According to \cite{mostow}, proof of 2.5, the $G$-module $C^1(G,M)$ is acyclic\footnote{The notation there for $C^1(G,M)$ is $F^0_0(G,M)$. One difference is that Mostow uses the complex $E^*(G,M)$ of equivariant homogeneous cochains in the definition of cohomology. However, the  isomorphism $E^n\rightarrow C^n$ that sends $f(g_0, g_1, \ldots, g_n)$ to
$f(1, g_1, g_1g_2, \ldots, g_1g_2\cdots g_n)$ identifies the two definitions. This is the usual comparison map one uses for discrete groups, which clearly preserves continuity. }, that is, $$H^i(G, C^1(G,M))=0$$ for $i>0$. Therefore, given a cocycle $c\in Z^i(G, M)$, there is an $F\in C^{i-1}(G, C^1(G,M))$ such that
its image $f \in C^{i-1}(G,N)$ is a cocycle and $dF=c$. Hence, $d(F^a-F)=c^a-c$. Also, by induction, there is a $ k_a\in C^{i-2}(G,N)$ such that $f^a-f=dk_a$ and $k_{ab}=(k_a)^b+k_b+dl$ for some $l\in C^{i-3}(G,N)$ (zero if $i=2$). Let $K_a=s\circ k_a $ and put $$h_a=F^a-F-dK_a.$$ Then the image of $h_a$ in $N$ is zero, so $h_a$ takes values in $M$, and $dh_a=c^a-c$. Now we check property (2). Note that
$$K_{ab}=s\circ k_{ab} =s\circ (k_a)^b+s\circ k_b+s\circ dl.$$ But $s\circ (k_a)^b-(s\circ k_a)^b$ and $s\circ dl-d(s\circ l)$ both have image in $M$. Hence, $K_{ab}=K_a^b+K_b+d(s\circ l)+m$ for some cochain $m\in C^{i-2}(G,M)$. From this, we deduce $$dK_{ab}=(dK_a)^b+dK_b+dm,$$from which we get
$$h_{ab}=F^{ab}-F-dK_{ab}=(F^a)^b-F^b+F^b-F-(dK_a)^b-dK_b-dm=(h_a)^b+h_b+dm.$$
\end{proof}
\section{Appendix B. Conjugation action on group cochains: categorical approach}

\author{\bf by Behrang Noohi}
\ms

In this section, an alternative and conceptual proof of Lemma 6.2 is outlined. Although not strictly necessary for the purposes of this paper, we believe that a functorial theory of secondary classes in group cohomology will be important in future developments. This point has also been emphasised to M.K. by Lawrence Breen. More details and elaborations will follow in a forthcoming publication by B.N.

\subsection{Notation}
In what follows $G$ is a group and $M$ is a left $G$-module. The action is denoted by
$\act{a}{m}$. The left conjugation action of $a\in G$ on $G$ is denoted
$\ad_a(x)=axa^{-1}$. We have an induced right action on 
$n$-cochains  $f \: G^{n} \to M$ given by
\[f^a(\mathbf{g}):=\act{a^{-1}}{(f(\ad_{a}\mathbf{g}))}.\]
Here, $\mathbf{g} \in G^n$ is an $n$-chain, and $\ad_a\mathbf{g}$ is defined componentwise.

In what follows, $[n]$ stands for the ordered set $\{0,1,\ldots,n\}$, viewed as
a category.

\subsection{Idea}
The above action on cochains respects the differential, hence passes to cohomology. It is 
well known that the induced action on cohomology is trivial. That is, given an
$n$-cocycle $f$ and any element $a\in G$, the difference $f^a - f$ is a coboundary.
In this appendix we explain how to construct an $(n-1)$-cochain
$h_{a,f}$ such that $d(h_{a,f})=f^a- f$. The construction, presumably well known,
uses standard ideas from simplicial homotopy theory. The general case of this construction,
as well as the missing proofs of some of the statements in this appendix will appear 
in a separate article.

Let $\mcG$ denote the one-object category (in fact, groupoid)  with morphisms $G$.
For an element $a \in G$, we have an action of $a$ on $\mcG$ which, by abuse of notation,
we will denote again by $\ad_a \: \mcG \to \mcG$; it fixes the unique object and acts
on morphisms by conjugation by $a$.

The main point in the construction of the cochain $h_{a,f}$ is that there is a 
``homotopy'' (more precisely, a natural transformation) $H_a$ from  the 
identity functor $\id \: \mcG \to \mcG$ to $\ad_a \: \mcG \to \mcG$.  
The homotopy between $\id$ and $\ad_a$  is given by the functor
$H_a \: \mcG\times[1] \to \mcG$ defined by 
      \[ H_a|_0= \id, \ \ H_a|_1=\ad_a, \text{ and } H_a(\iota) = a^{-1}.\]
It is useful to visualise the category $\mcG\times[1]$ as
    \[\xymatrix@C=30pt@R=12pt@M=3pt{ 0 \ar[r]^{\iota} \ar@(ur,ul)[]_G 
                          & 1 \ar@(ur,ul)[]_G}.\]

\subsection{Cohomology of categories}
We will use multiplicative notation for morphisms in a category, namely,
the  composition of $g \: x \to y$ with $h \: y \to z$ is denoted  $gh \: x \to z$.

Let $\mcC$ be a small category and $M$ a left $\mcC$-module, that is, a functor 
$M\: \mcC^{\text{op}} \to \mathbf{Ab}$, $ x \mapsto M_x$, to the category of abelian groups
(or your favorite linear category). 
Note that when $\mcG$ is as above, this is  nothing but a left $G$-module 
in the usual sense. For an arrow $g \: x \to y$ in $\mcC$, we denote the induced map
$M_y \to M_x$ by $m \mapsto \act{g}{m}$.

Let $\mcC^{[n]}$ denote the set of all $n$-tuples $\bfg$ of composable arrows in  $\mcC$,
    \[\bfg \ = \ \bullet \xrightarrow{g_1} \bullet \xrightarrow{g_2} 
           \cdots \xrightarrow{g_n}\bullet.\]
We refer to such a $\bfg$ as an  {\bf $n$-cell} in $\mcC$; this is the same thing
as a functor $[n] \to \mcC$,
which we will denote, by abuse of notation, again by $\bfg$. 

An {\bf $n$-chain} in $\mcC$ is an element in the free abelian
group $\oC_n(\mcC,\mathbb{Z})$ generated by the set $\mcC^{[n]}$ of $n$-cells.
For an $n$-cell $\bfg$ as above, we let $s\bfg \in \Ob\mcC$ denote  the source of $g_1$.

By an {\bf $n$-cochain} on $\mcC$ with values in $M$ we mean a map
$f$ that assigns to any $n$-cell $\bfg \in \mcC^{[n]}$ an element in $M_{s\bfg}$.
Note that, by linear extension, we can evaluate $f$ on any $n$-chain 
in which all $n$-cells share a common source point.

The $n$-cochains form an abelian group  $\oC^n(\mcC,M)$.
The {\bf cohomology} groups $\oH^n(\mcC,M)$, $n\geq 0$, are defined using the  
cohomology complex $\oC^{\bullet}(\mcC,M)$: 
   \[  0 \xrightarrow{}\oC^0(\mcC,M) \xrightarrow{d}  \oC^1(\mcC,M) \xrightarrow{d}
   \cdots \xrightarrow{d} \oC^n(\mcC,M) \xrightarrow{d} \oC^{n+1}(\mcC,M) 
   \xrightarrow{d} \cdots\]
where the differential 
\[d \:\oC^n(\mcC,M) \to \oC^{n+1}(\mcC,M)\] 
is defined by
  \begin{multline*}
   df(g_1,g_2,\ldots,g_{n+1})   =  \act{g_1}(f(g_2,\ldots,g_{n+1})) +
   \underset{1\leq i \leq n}\sum(-1)^if(g_1,\ldots,g_ig_{i+1},\ldots,g_{n+1}) \\
   + (-1)^{n+1}f(g_1,g_2,\ldots,g_{n}).
  \end{multline*}

A left $G$-module $M$ in the usual sense gives rise to a left module on $\mcG$, 
which we denote again by $M$. We sometimes denote $\oC^{\bullet}(\mcG,M)$
by $\oC^{\bullet}(G,M)$. Note that the corresponding cohomology groups coincide with
the group cohomology  $\oH^n(G,M)$.

The cohomology complex $\oC^{\bullet}(\mcC,M)$ and the cohomology groups $\oH^n(\mcC,M)$
are functorial in $M$. They are also functorial in $\mcC$ in the following sense. 
A functor $\varphi \: \mcD \to \mcC$ gives rise to a $\mcD$-module  
$\varphi^*M:=M\circ \varphi \: \mcD^{op} \to \mathbf{Ab}$. We have a  map of complexes 
  \begin{equation}\label{Eq:1} \varphi^* \: \oC^{\bullet}(\mcC,M) 
     \to \oC^{\bullet}(\mcD,\varphi^*M), 
  \end{equation}
which gives rise to the maps
  \[ \varphi^* \: \oH^{n}(\mcC,M) \to \oH^{n}(\mcD,\varphi^*M) \]
on cohomology, for all $n\geq 0$.

\subsection{Definition of the cochains $h_{a,f}$}
The flexibility we gain by working with chains on general categories allows us
to import standard ideas from topology to this setting. The following definition
of the cochains $h_{a,f}$ is an imitation of a well known construction in topology.

Let $f \in \oC^{n+1}(G,M)$ be an $(n+1)$-cochain, and $a \in G$ an element. Let
$H_a \: \mcG\times[1] \to \mcG$ be the corresponding natural
transformation. We define $h_{a,f} \in \oC^{n}(G,M)$ by
      \[h_{a,f}(\bfg)=f(H_a(\bfg\times[1])).\]
Here, $\bfg \in \mcC^{[n]}$ is an $n$-cell in $\mcG$, so $\bfg\times[1]$ is an  
$(n+1)$-chain in $\mcG\times[1]$, namely, the cylinder over $\bfg$.

To be more precise, we are using the notation $\bfg\times[1]$ for the image of 
the fundamental class of $[n]\times [1]$ in $\mcG\times[1]$ under the 
functor $\bfg \times [1] \: [n]\times [1] \to \mcG \times [1]$. 
We visualize $[n]\times [1]$ as
     \[\xymatrix@C=10pt@R=12pt@M=6pt{ 
         (0,1) \ar[r]        & (1,1) \ar[r]         & \ar[r] \cdots &(n,1)\\
         (0,0) \ar[r] \ar[u] & (1,0) \ar[r] \ar[u]  & \ar[r] \cdots &
         (n,0)   \ar[u] }\]
Its fundamental class  is the alternating sum of the $(n+1)$-cells
     \[\xymatrix@C=10pt@R=12pt@M=6pt{ 
                       & &   (r,1) \ar[r]               & \cdots  \ar[r] &(n,1)\\
         (0,0) \ar[r]  & \cdots  \ar[r] &(r,0) \ar[u]   &  &   }   \]
in  $[n]\times [1]$, for $0\leq r \leq n$. Therefore, 
   \begin{equation}
     h_{a,f}(\bfg) = \underset{0\leq r \leq n}{\sum} (-1)^r
     f(g_1,\ldots,g_r,a^{-1},\ad_a{g_{r+1}},\ldots,\ad_a{g_{n}}).
   \end{equation}

The following proposition can be proved using a variant of Stokes' 
formula for cochains.

\begin{prop}{\label{P:homotopy}}
The graded map $h_{-,a}\: \oC^{\bullet+1}(G, M) \to \oC^{\bullet}(G,M)$ 
is a chain homotopy between the chain maps 
     \[\id,(-)^a \: \oC^{\bullet}(G,M) \to \oC^{\bullet}(G,M).\]
That is,
     \[ h_{a,df} + d(h_{a,f}) = f^a - f\]
for every $(n+1)$-cochain $f$. In particular, if $f$ is an $(n+1)$-cocycle, 
then  $d(h_{a,f}) = f^a - f$.  
\end{prop}

\subsection{Composing natural transformations}
Given an $(n+1)$-cochain  $f$, and elements $a,b \in G$, we can construct three 
$n$-cochains:
$h_{a,f}$, $h_{b,f}$ and $h_{ab,f}$. A natural question to ask is whether
these three cochains satisfy a cocycle condition. It turns out that the answer
is yes, but only up to a coboundary $dh_{a,b,f}$. Below we explain how 
$h_{a,b,f}$ is constructed. In fact, we construct cochains $h_{a_1,\ldots,a_k,f}$, for
any $k$ elements $a_i \in G$, $1\leq i \leq k$, and study their relationship.

Let $f \in \oC^{n+k}(G,M)$ be an $(n+k)$-cochain. Let 
$\bfa=(a_1,\ldots,a_k)\in G^{\times k}$. Consider the category $\mcG\times [k]$,
      \[\xymatrix@C=30pt@R=12pt@M=3pt{ 0 \ar[r]^{\iota_0} \ar@(ur,ul)[]_G 
          & 1 \ar@(ur,ul)[]_G    \ar[r]^(0.37){\iota_1} 
          & \ \ \cdots \ \ \ar[r]^(0.6){\iota_{k-1}} & k  \ar@(ur,ul)[]_G.}\]

Let $H_{\bfa} \: \mcG\times [k] \to \mcG$ be the functor such that 
$\iota_i \mapsto a_{k-i}^{-1}$ and $H_{\bfa}|_{\{0\}}=\id_G$. 
(So, $H_{\bfa}|_{\{k-i\}}=\ad_{a_{i+1}\cdots a_k}$.) Define
$h_{\bfa,f} \in \oC^{n}(G,M)$ by
    \begin{equation}{\label{Eq:haf}}
      h_{\bfa,f}(\bfg)=f(H_{\bfa}(\bfg\times[k])).
    \end{equation}
Here, $\bfg \in \mcC^{[n]}$ is an $n$-cell in $\mcG$, so $\bfg\times[k]$ is an  
$(n+k)$-chain in $\mcG\times[k]$.

To be more precise, we are using the notation $\bfg\times[k]$ for  the image of 
the fundamental class of $[n]\times [k]$ in $\mcG\times[k]$ under the 
functor $\bfg \times [k] \: [n]\times [k] \to \mcG \times [k]$. We visualize
$[n]\times [k]$ as
     \[\xymatrix@C=10pt@R=12pt@M=6pt{ 
         (0,k) \ar[r]        & (1,k) \ar[r]         & \ar[r] \cdots &  (n,k)        \\
           \vvdots\ar[u]     &   \vvdots \ar[u]     &               & \vvdots\ar[u] \\
         (0,1) \ar[r] \ar[u] & (1,1) \ar[r] \ar[u]  & \ar[r] \cdots &  (n,1) \ar[u] \\
         (0,0) \ar[r] \ar[u] & (1,0) \ar[r] \ar[u]  & \ar[r] \cdots &  (n,0) \ar[u] }\]
Its fundamental class is the $(n+k)$-chain
       \[\underset{P}\sum (-1)^{|P|}P,\]
where $P$ runs over (length $n+k$) paths starting from $(0,0)$ and ending in 
$(n,k)$. Note that such paths correspond to $(k,n)$ shuffles; $|P|$ stands for 
the parity of the shuffle (which is the same as the number of squares above the 
path in the $n\times k$ grid).

The most economical way to describe the relations between various $h_{\bfa,f}$ 
is in terms of the cohomology complex of the right module 
    \[\bbM^{\bullet}:=\uHom\left(\oC^{\bullet}(G,M),\oC^{\bullet}(G,M)\right).\]
Here, $\uHom$ stands for the enriched hom in the category of chain complexes, 
and the right action of $G$ on $\bbM^{\bullet}$ 
is induced from the right action $f \mapsto f^a$ of $G$ on the
$\oC^{\bullet}(G,M)$ sitting on the right. 
The differential on $\bbM^{\bullet}$ is defined by
    \[d_{\bbM^{\bullet}}(u)=(-1)^{|u|} u\circ d_{\oC^{\bullet}(G,M)}-
           d_{\oC^{\bullet}(G,M)}\circ u,\]
where $|u|$ is the degree of the homogeneous $u \in \oC^{\bullet}(G,M)$.

Note that, for every $\bfa \in G^{\times k}$, we have $h_{\bfa,f} \in \bbM^{-k}$.
This defines a $k$-cochain  on $G$ of degree $-k$ with values in $\bbM^{\bullet}$,
    \[h^{(k)}\: \bfa \ \mapsto h_{\bfa,-}, \ \bfa \in G^{\times k}.\]
We set $h^{(-1)}:=0$.  Note that $h^{(0)}$ is the element in $\bbM^{0}$
corresponding to the identity map $\id \: \oC^{\bullet}(G,M) \to \oC^{\bullet}(G,M)$.
  
The relations between various $h_{\bfa,f}$ can be packaged in a simple differential
relation. As in the case $k=0$ discussed in Proposition \ref{P:homotopy}, 
this proposition can be proved using a variant of Stokes' formula for cochains.

\begin{prop}{\label{P:relations}}
For every $k \geq -1$, we have $d_{\bbM^{\bullet}}(h^{(k+1)})=d(h^{(k)})$.
\end{prop}

In the above formula,  the term $d_{\bbM^{\bullet}}(h^{(k+1)})$ means that we apply 
$d_{\bbM^{\bullet}}$ to the values (in $\bbM^{\bullet}$) of the cochain $h^{(k+1)}$.
The differential on the right hand side of the formula is the differential of the
cohomology complex $\oC^{\bullet}(G,\bbM^{\bullet})$ of the (graded) right $G$-module  
$\bbM^{\bullet}$.
 
More explicitly, let $f\in \oC^{n+k}(G,M)$ be an $(n+k)$-cochain. Then, 
Proposition \ref{P:relations} states that, for every $\bfa \in G^{\times (k+1)}$, 
we have the following equality of $n$-cochains:
  \begin{eqnarray*}
    (-1)^{(k+1)} h_{a_1,\ldots,a_{k+1},df}-dh_{a_1,\ldots,a_{k+1},f}
       = & h_{a_2,\ldots,a_{k+1},f} + \\
      & \underset{1\leq i \leq k}\sum(-1)^ih_{a_1,\ldots,a_ia_{i+1},\ldots,a_{k+1},f} + \\
      &  (-1)^{k+1}h_{a_1,\ldots,a_{k},f}^{a_{k+1}}.
  \end{eqnarray*}

\begin{cor}{\label{C:left}} 
Let $f\in \oC^{n+k}(G,M)$ be an $(n+k)$-cocycle. Then, for every
$\bfa \in G^{\times (k+1)}$, the $n$-cochain 
    \[h_{a_2,\ldots,a_{k+1},f} 
         + \underset{1\leq i \leq k}\sum(-1)^ih_{a_1,\ldots,a_ia_{i+1},\ldots,a_{k+1},f}  
         + (-1)^{k+1}h_{a_1,\ldots,a_{k},f}^{a_{k+1}}\]
is a coboundary. In fact, it is the coboundary of $-h_{a_1,\ldots,a_{k+1},f}$.       
\end{cor}

\begin{ex}
Let us examine Corollary \ref{C:left} for small values of $k$. 
 \begin{itemize}
   \item[i)] For $k=0$, the statement is that, for every cocycle $f$,
    $f-f^a$ is a coboundary. In fact, it is the coboundary of $-h_{f,a}$.
    We have already seen this in Proposition \ref{P:homotopy}.
   \item[ii)] For $k=1$, the statement is that, for every cocycle $f$, the cochain
       \[h_{b,f}-h_{ab,f}+h_{a,f}^b \]  
    is a coboundary. In fact, it is the coboundary of $-h_{a,b,f}$.
 \end{itemize}
\end{ex}

\subsection{Explicit formula for $h_{a_1,\ldots,a_k,f}$}{\label{SS:Eplicit}}  
Let $f \: G^{\times (n+k)} \to M$ be an $(n+k)$-cochain, and 
$\bfa:=(a_1,a_2,\ldots,a_{k})  \in G^{\times k}$. Then, by (\ref{Eq:haf}),
the effect of  the $n$-cochain $h_{a_1,\ldots,a_k,f}$ on 
an $n$-tuple $\bfx:=(x_0,x_1,\ldots,x_{n-1}) \in G^{\times n}$
is given by:
     \[h_{a_1,\ldots,a_k,f}(x_0,x_1,\ldots,x_{n-1})
           =\underset{P}\sum (-1)^{|P|}f(\bfx^{P}),\]
where $\bfx^{P}$ is the $(n+k)$-tuple obtained by the following procedure.

Recall that $P$ is a path from $(0,0)$ to $(n,k)$ in the $n$ by $k$ grid. The
$l^{\text{th}}$ component $\bfx^{P}_l$ of $\bfx^{P}$ is determined by the 
$l^{\text{th}}$ segment on the path $P$. Namely, suppose that the coordinates 
of the starting point of this segment are  $(s,t)$. Then, 
       \[\bfx^{P}_l=a_{k-t}^{-1}\] 
if the segment is vertical, and 
       \[\bfx^{P}_l=(a_{k-t+1}\cdots a_k)x_s(a_{k-t+1}\cdots a_k)^{-1},\] 
if the segment is horizontal. Here, we use the convention that $a_0=1$.

The following example helps visualize $\bfx^{P}$:
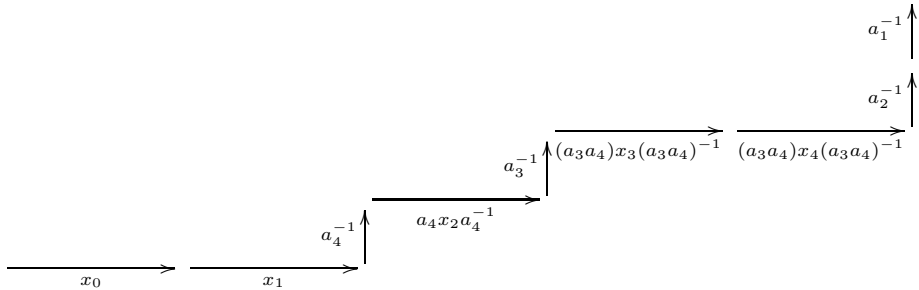
   \[\xymatrix@C=63pt@R=20pt@M=3pt{ 
              &   &  &  &  &       \\
              &   &  &  &  &  \ar[u]^{a_1^{-1}}       \\
              &   &  &  \ar[r]_{(a_3a_4)x_3(a_3a_4)^{-1}}            &  
              \ar[r]_{(a_3a_4)x_4(a_3a_4)^{-1}} & \ar[u]^{a_2^{-1}}\\
              &   &  \ar[r]_{a_4x_2a_4^{-1}}    &  \ar[u]^{a_3^{-1}} &  &\\
              \ar[r]_{x_0} & \ar[r]_{x_1}       &  \ar[u]^{a_4^{-1}} &  &  &  &}\]
The corresponding term is
     \[-f(x_0,x_1,a_4^{-1},a_4x_2a_4^{-1},
         a_3^{-1},(a_3a_4)x_3(a_3a_4)^{-1},(a_3a_4)x_4(a_3a_4)^{-1},
                                       a_2^{-1},a_1^{-1}).\]
The sign of the path is determined by the parity  of the number of squares 
in the $n$ by $k$ grid that sit above the path $P$ (in this case $15$).

\section*{Acknowledgements}
I (M.K) owe a tremendous debt of gratitude to many people for conversations, communications, and tutorials about a continuous stream of facts and theories that I barely understand even now. These include John Baez, Alan Barr, Bruce Bartlett, Jean Bellissard, Philip Candelas, John Coates, Hee-Joong Chung, Tudor Dimofte,  Dan Freed, Sergei Gukov, Jeff Harvey, Yang-Hui He, Lars Hesselholt, Mahesh Kakde, Kazuya Kato, Dohyeong Kim, Philip Kim, Kobi Kremnitzer, Julien March\'e, Behrang Noohi, Xenia de la Ossa,  Jaesuk Park, Jeehoon Park, Alexander Schekochihin, Alexander Schmidt, Urs Schreiber, Graeme Segal,  Adam Sikora, Peter Shalen, Romyar Sharifi, Junecue Suh, Kevin Walker, Andrew Wiles, and Hwajong Yoo. 

\end{flushleft}


\begin{thebibliography}{30}
\bibitem{atiyah}{Atiyah, Michael Topological quantum field theories. Inst. Hautes \'Etudes Sci. Publ. Math. No. 68 (1988), 175--186 (1989). 

}
\bibitem{BW}{Bates, Sean; Weinstein, Alan
Lectures on the geometry of quantization. 
Berkeley Mathematics Lecture Notes, 8. American Mathematical Society, Providence, RI; Berkeley Center for Pure and Applied Mathematics, Berkeley, CA, 1997.  }
\bibitem{CK}{ Coates, John; Kim, Minhyong Selmer varieties for curves with CM Jacobians. Kyoto J. Math. 50 (2010), no. 4, 827--852. }
\bibitem{CKPY} {Chung, Hee-Joong; Kim, Dohyeong; Park, Jeehoon; Yoo, Hwajong The Glueing Formula for Arithmetic Chern-Simons Theory. (In preparation)}
\bibitem{deninger}{Deninger, Christopher
Analogies between analysis on foliated spaces and arithmetic geometry. Groups and analysis, 174-190, 
London Math. Soc. Lecture Note Ser., 354, Cambridge Univ. Press, Cambridge, 2008. }
\bibitem{DW}{ Dijkgraaf, Robbert; Witten, Edward Topological gauge theories and group cohomology. Comm. Math. Phys. 129 (1990), no. 2, 393--429.}

\bibitem{FF}{Fargues, Laurent; Fontaine, Jean-Marc Vector bundles and p-adic Galois representations. Fifth International Congress of Chinese Mathematicians. Part 1, 2, 77--113, AMS/IP Stud. Adv. Math., 51, pt. 1, 2, Amer. Math. Soc., Providence, RI, 2012.}
\bibitem{FQ}{Freed, Daniel S.; Quinn, Frank
Chern-Simons theory with finite gauge group. 
Comm. Math. Phys. 156 (1993), no. 3, 435--472. }
 \bibitem{FK}{Fukaya, Takako; Kato, Kazuya A formulation of conjectures on p-adic zeta functions in noncommutative Iwasawa theory. Proceedings of the St. Petersburg Mathematical Society. Vol. XII, 1–85, Amer. Math. Soc. Transl. Ser. 2, 219, Amer. Math. Soc., Providence, RI, 2006.}
 \bibitem{huber}{Huber, Annette  Poincar\'e duality for p-adic Lie groups. Archiv der Mathematik
December 2010, Volume 95, Issue 6, pp 509--517}

\bibitem{KW}{Kapustin, Anton; Witten, Edward
Electric-magnetic duality and the geometric Langlands program. 
Commun. Number Theory Phys. 1 (2007), no. 1, 1--236. }

\bibitem{kato}{ Kato, Kazuya Lectures on the approach to Iwasawa theory for Hasse-Weil L-functions via $B_{DR}$. I. Arithmetic algebraic geometry (Trento, 1991), 50--163, Lecture Notes in Math., 1553, Springer, Berlin, 1993. }
\bibitem{KH}{ Hornbostel, Jens; Kings, Guido On non-commutative twisting in \'etale and motivic cohomology. Ann. Inst. Fourier (Grenoble) 56 (2006), no. 4, 1257--1279. }
\bibitem{kim1}{ Kim, Minhyong The motivic fundamental group of $\P^1\setminus \{0,1,\infty\} $ and the theorem of Siegel. Invent. Math. 161 (2005), no. 3, 629--656.}
\bibitem{kim2} {Kim, Minhyong p-adic L-functions and Selmer varieties associated to elliptic curves with complex multiplication. Ann. of Math. (2) 172 (2010), no. 1, 751--759.}
\bibitem{kim3}{ Kim, Minhyong Massey products for elliptic curves of rank 1. J. Amer. Math. Soc. 23 (2010), no. 3, 725--747. }
\bibitem{kim4}{Kim, Minhyong Tangential localization for Selmer varieties. Duke Math. J. 161 (2012), no. 2, 173–199.} 
\bibitem{KM}{  Knudsen, Finn Faye; Mumford, David The projectivity of the moduli space of stable curves. I. Preliminaries on "det'' and "Div''. Math. Scand. 39 (1976), no. 1, 19--55. }
\bibitem{mazur}{Mazur, Barry
Notes on etale cohomology of number fields. 
Ann. Sci. \'Ecole Norm. Sup. (4) 6 (1973), 521--552 (1974). }
\bibitem{mazur2}{Mazur, Barry  Remarks on the Alexander polynomial. Unpublished notes.}
\bibitem{morishita}{Morishita, Masanori
Knots and primes. 
An introduction to arithmetic topology. Universitext. Springer, London, 2012.}
\bibitem{mostow}{Mostow, George D. Cohomology of Topological Groups and Solvmanifolds. 
Annals of Mathematics(2), Vol. 73, No. 1 (Jan., 1961), pp. 20--48}
\bibitem{NSW}{Neukirch, J\"{u}rgen; Schmidt, Alexander; Wingberg, Kay Cohomology of number fields. Second edition. Grundlehren der Mathematischen Wissenschaften, 323. Springer-Verlag, Berlin, 2008.}
\bibitem{witten}{ Witten, Edward Quantum field theory and the Jones polynomial. Comm. Math. Phys. 121 (1989), no. 3, 351--399.}



\end{thebibliography}
\end{document}